\documentclass[11pt,a4paper]{article}
\usepackage[latin1]{inputenc}
\usepackage{amsmath, amsfonts, amssymb, amsthm, graphicx,color}

\renewcommand{\k}{\kappa}
\renewcommand{\i}{\mathit{i}}

\renewcommand{\L}{\ell}
\renewcommand{\ss}{\mathsf {S}}
\newcommand{\intr}[1]{#1^\circ}
\newcommand{\clos}[1]{\overline{#1}}
\newcommand {\leng}[1] {\mathsf{Length}(#1)}

\def \conv {\mathsf{Conv}}

\newtheorem{thm}{Theorem}
\newtheorem{lm}{Lemma}

\newtheorem{qst}{Question}
\theoremstyle{definition}
\newtheorem{rem}{Remark}
\newtheorem{definition}{Definition}

\title{Non-smooth convex caustics for Birkhoff billiard.}
\author{Maxim Arnold\thanks{University of Texas at Dallas, Richardson, TX, 
USA} \and Misha Bialy\thanks{Tel Aviv University, Tel Aviv, Israel}.}

\begin{document}

	\maketitle

	\begin{abstract}
		This paper is devoted to the examination of the properties of the 
		string 
		construction for the Birkhoff billiard. Based on purely geometric 
		considerations, string construction is suited to provide a table for the 
		Birkhoff billiard, having the prescribed caustic. Exploiting this 
		framework together with the  properties of convex caustics, we give a 
	geometric proof of a result by Innami first proved in 2002 by means of 
	Aubry-Mather theory. In the second part of the paper
	we show that applying the string construction
	one can find a new collection of examples of $C^2$-smooth
	convex billiard tables with a non-smooth convex caustic.
\end{abstract}

	\section{Introduction.}
	Let $\Gamma$ be a  simple closed $C^1$-smooth convex curve in the 
	Eucledian plane.
	We consider Birkhoff billiard inside $\Gamma$. This simple dynamical 
	system creates many 
	geometric and dynamical questions and reflects many difficulties 
	appearing in general Hamiltonian systems. Reader may refer to any  
	textbook among the wide variety written on the subject (e.g.  
	\cite{Katok_Strelcyn}, \cite{Kozlov}, \cite{Forni_Mather}, 
	\cite{Tabachnikov}). 
	 
	In the present paper we will use the following  non-standard
	notations: the interior of the set bounded by  
	simple closed curve $\gamma$ will be denoted by  
	$\intr{\gamma}$,   
	while $\clos{\gamma}$ denotes the compact 
	$\intr{\gamma}\cup\gamma$.  Length of the curve is denoted by 
	$\leng{\gamma}$. Convex hull of $\gamma$ is denoted by 
	$\conv(\gamma)$. The 
	following definition of convex caustics is used in this paper:
\begin{definition} 
Simple closed curve  $\gamma\subset \intr{\Gamma}$ 
is called \emph{convex caustic} for $\Gamma$
if $\bar\gamma$ is a convex set and any 
supporting line to $\bar\gamma$ remains a supporting line to 
$\bar{\gamma}$ after 
the 
billiard reflection in $\Gamma$.  
\end{definition}
 Every convex caustic $\gamma$ corresponds  to 
the invariant curve 
$r_{\gamma}$ of the billiard ball map. Curve $r_\gamma\subset 
\mathbb{R}_+\times \mathbb{S}^1$ consists of all supporting lines to 
$\gamma$. This curve winds once around the phase cylinder and therefore 
is called rotational. We shall denote its rotation number by 
$\rho_{\gamma}$.

In the original Birkhoff paper \cite{Birkhoff} there was posed a conjecture 
that the existence of a continuous set of caustics, being very restrictive 
property, 
actually provide  an extreme rigidity on the shape of curve $\Gamma$. First 
result in this direction was achieved in \cite{Bialy}. 
Our paper is motivated by recent progress in the Birkhoff conjecture 
solution 
achieved in
\cite{Kaloshin_Avila,Kaloshin_Sorrentino}. The crucial assumption in these 
papers
consists in the existence of convex caustics such that the rotation numbers 
of the corresponding invariant curves form a rational sequence in the 
interval 
$(0;\tfrac 13]$, converging to $0$. It seems natural to compare such 
result 
with one proved by 
N. Innami \cite{Innami}.
\begin{thm}[Innami (2002), \cite{Innami}]\label{thm:Innami}
Assume that there exists a sequence of convex caustics $\gamma_n$ 
inside $\Gamma$ such that the rotation numbers $\rho_n$ of the 
corresponding invariant curves tend to $\frac{1}{2}$. Then $\Gamma$
is an ellipse.
\end{thm}

Originally, Innami's arguments were based on the Aubry-Mather variational 
theory. In the next section we present a simple geometric proof using 
string construction. Yet, it remains a challenging question if one can prove 
more general 
statement relaxing the requirement of convexity of the caustics.

Let us remind the string construction framework. Given a convex 
compact set $\bar{\gamma}$  bounded by $\gamma$, and a number 
$S>\leng{\gamma}$ define the 
curve $\Gamma$  as a union of those points $P$ that the \emph{cap-body}
$\conv(P\cup \clos{\gamma})$
has the boundary of the length $S$.
Geometrically such a construction gives the set of all points traversed by 
the tip of non-elastic string of length $S>\leng{\gamma}$ wrapped 
around $\gamma$ and stretched to the very extent. Curve 
$\Gamma$ provided by such construction has $\gamma$ 
as its billiard caustic.  We shall refer to $S$ as a string parameter of the 
caustic. A closely related so-called Lazutkin parameter is defined as
$
L=S-\leng{\gamma}.
$

The string construction is widely known and can be easily proved to 
provide $\Gamma$ for smooth enough $\gamma$. In fact it remains valid 
also in more general case as it is stated in the following theorem.
\begin{thm}[Stall (1930), \cite{Stall}; Turner (1982),
\cite{turner}]\label{thm:string}$\phantom{1}$
\begin{enumerate}
	\item For a given compact convex  set $\bar\gamma$ and for every 
$S>\leng{\gamma}$ the string construction
determines a $C^1$-smooth convex closed curve $\Gamma$ such that 
$\gamma$ 
is a billiard caustic for $\Gamma$.

\item If   $\gamma$ is a convex billiard caustic for $C^1$ curve 
$\Gamma$ then $\Gamma$ can be obtained from $\gamma$ by  
the string construction for some $S$. 
\end{enumerate}
\end{thm}

Let us emphasize that the string construction is highly non-explicit and 
difficult for calculations. Very important consequence of KAM theory, 
proved by Lazutkin \cite{Lazutkin, Lazutkin_book} and Douady 
\cite{Douady}, states the existence of convex caustics 
near the boundary  of sufficiently smooth (at least $C^6$) billiard table. On 
the other hand, applying string construction to the triangle, one 
gets 
billiard table which is piecewise $C^2$ with jumps of the 
curvature and hence by \cite{Hubacher} can not have caustics near the 
boundary. 

The scenario of destruction of caustics when one moves away 
from the boundary towards the interior could be understood in principle by 
the analogy with wave front propagation inside a convex curve 
(\cite{Forni_Mather}). Take for 
example the ellipse and consider the wave fronts as on the 
famous picture 
(\cite[Fig.36]{Arnold}).
For small distances the fronts remain smooth, but starting from some 
critical 
value they start to develop singularities. However, nobody saw such a 
bifurcation in 
practice for caustics of convex billiards due to the lack of integrable 
examples. On the other hand, non-convex caustics exist for instance for 
convex bodies of constant width, and  were studied in 
\cite{Knill}.  

Motivated by the above discussion, the natural question about the 
existence  of non-smooth convex caustics arises. 
More generally, it is natural to study how irregular the convex caustic can 
be. In \cite{Fetter} a billiard table of class $C^2$ was constructed which has 
a 
caustic of regular hexagon.
In the present paper we were able to construct whole functional family of 
the examples of 
$C^2$ billiard tables having non-smooth  convex caustics.
\begin{thm}
 There exist a one-parametric family of strictly convex non-smooth 
 compact sets $\bar{\gamma}$ and  
 the values of the string parameter $S$  
 such that the curves $\Gamma$ obtained by the string construction 
 are $C^2$-smooth. 
\end{thm}


We will use the following geometric idea (we will use complex notations 
$x+iy$ for points $(x,y)$ in the 
plane).  Start with a curve $\gamma_0(t): [-1,1]\to 
\mathbb{C}$ such that $\gamma_0(-1)=A=-1-i$, 
$\gamma_0(1)=iA=1-i$ and $\gamma_0(t)$ is symmetric with respect 
to vertical axis (i.e. $i\gamma_0(-t)=\overline{i\gamma_0(t)}$ (see Fig. 
\ref{fig:table_intro}). Construct $\gamma$ 
as a concatenation of $\{ i^k\gamma_0\}_{k=0}^3$.
Parametrize $\gamma$ by the arc-length parameter $s$ and choose the 
initial point in such a way that $\gamma(0)=A$. We will denote the total 
length of $\gamma$ by $4\ss$. Then $\gamma(\ss)=iA$.

	\begin{figure}[!hbt]
	\centering
	\includegraphics[width=0.45\textwidth]{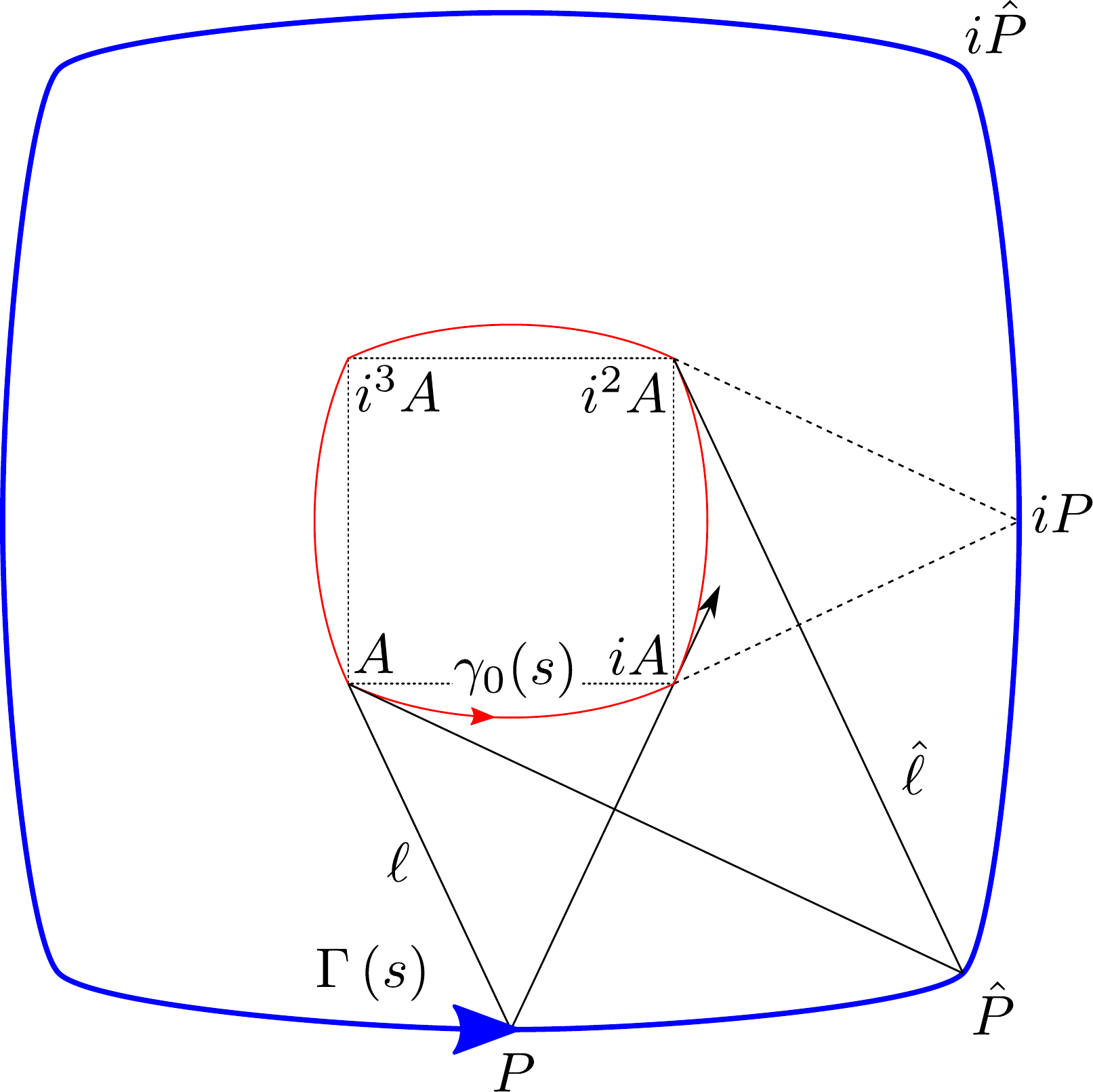}
	\caption{Switched caustic string construction.}\label{fig:table_intro}
\end{figure}

 Main idea is to 
choose the curve $\gamma$ and string parameter $S$ in such a way 
that the string 
construction will have the following properties: 

\begin{itemize}
	\item At the beginning (point $P$ on Fig.\ref{fig:table_intro}), left part 
	$AP$ of the string remains fixed at point $A$ while the right part of the 
	string unwind from the arc $(\widehat{iA, i^2A})$. 
	\item At the moment when the left part of the string became tangent to 
	$\gamma$ at the point $A$ (this corresponds to the point $\hat{P}$ on 
	$\Gamma$)  right part reaches the point $i^2A$ and remains fixed after 
	that. We will call this moment the \emph{switching of the first kind}.
	\item While the left part of the string winds around the arc 
	$(\widehat{A,iA})$ the right 
	part remains fixed at $i^2A$ (see Fig. \ref{fig:table_intro}) till the 
	moment 
	when the vertex of the string reaches the point $iP$. We will 
	call this \emph{switching of the second 
	kind}.
	\item $D_4$ symmetry provides the whole picture. 
	\end{itemize}

Let us reemphasize, that the string construction being non-explicit 
procedure, typically does not provide any analytic expression for the table 
$\Gamma$ from given $\gamma$. In the example \cite{Fetter}, the 
construction is made explicit by fixing two end-points on the string. 
Disadvantage of such situation is the complete loss of any flexibility, since 
the corresponding table may consist only of the elliptic arcs. In the current 
paper we propose another, more flexible yet explicit construction, fixing 
only one end-point of the string and allowing another point to slide along 
the given curve $\gamma$.  


\subsection*{Structure of the paper.} In the next section we will provide 
geometric arguments for the proof of the Theorem \ref{thm:Innami}.
Section 3 is devoted to the construction of the $C^2$ tables with 
non-smooth caustics.  In 
Section 4 we will pose some open questions arising in our considerations. 
\subsection*{Acknowledgments.} MB is thankfull to the participants of the course  "Billiards" given in Tel Aviv University for very usefull discussions and ideas.
MB was supported by ISF 162/15.
		\section{Geometric proof of Innami's result.\label{sec:Innami}} 
We will start with the following simple remarks.
\begin{rem}\label{rem:empty}
	If billiard in $\Gamma$ has a convex caustic $\gamma$ with 
	$\intr{\gamma}=\emptyset$ 
	then $\Gamma$ is either an ellipse or a 
	circle. 
\end{rem}
Indeed, condition $\intr{\gamma}=\emptyset$ for convex $\gamma$
means 
that $\gamma$ is either a 
point or a segment. The rest follows from the string construction.

	\begin{rem}\label{rem:different}
		If convex caustic $\gamma$ has
		non-empty interior, then every supporting line to $\bar\gamma$ 
		after 
		reflection in 
		$\Gamma$ at point $P$ becomes second 
		supporting line to $\bar\gamma$ from $P$. Recall that for any point 
		$P$ and for any convex 
		body $C$ there exist exactly two supporting lines to $C$ passing 
		through $P$.
	\end{rem}
 Assume that there exists a supporting line to $\bar\gamma$ which 
 is reflected to itself. Then, by 
 continuity, since $\gamma$ has 
 non-empty interior, all lines must behave like this. Therefore, all lines 
 tangent to $\bar{\gamma}$ are diameters, but 
 then for any 
 point $P\in\Gamma$
	there are two diameters passing through $P$ which is not possible.
	Finally, we get:
	\begin{lm}\label{lm:ball}
	Let	$\gamma$ be a convex caustic for $\Gamma$.
	Then  $\intr{\gamma}\ne\emptyset$  if and only if the rotation 
	number of the corresponding invariant curve is strictly less then 
	$\frac{1}{2}$.
	\end{lm}
\begin{proof}
	\begin{figure}[hbt]
		\centering
		\includegraphics[width=0.4\textwidth]{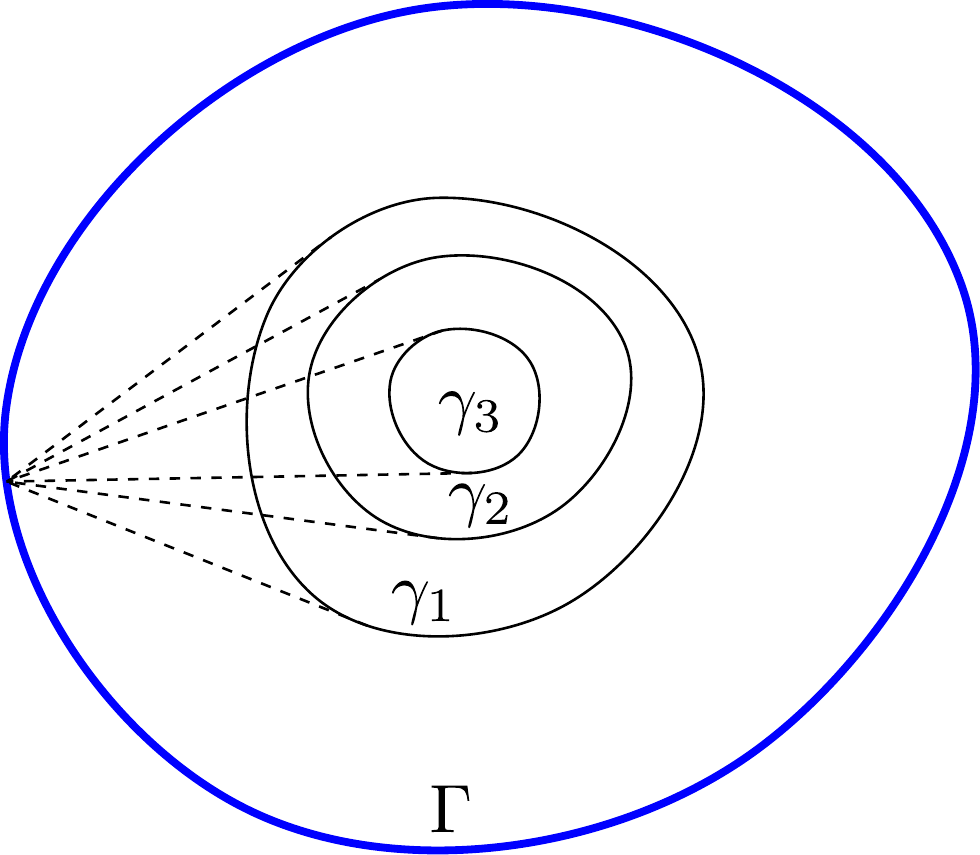}\qquad
		\includegraphics[width=0.4\textwidth]{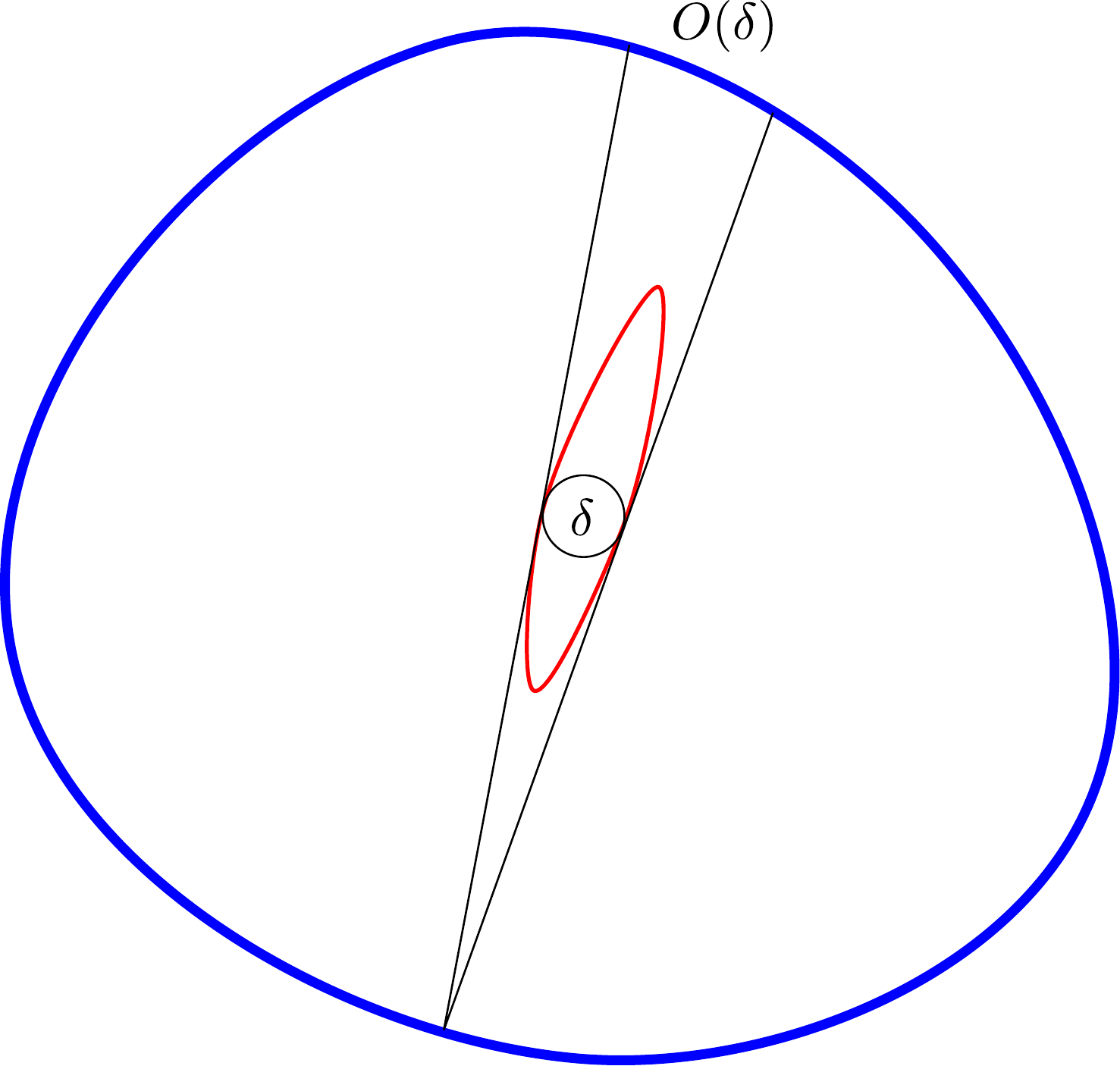}
		\caption{Left: Family of nested convex caustics with decreasing string 
			parameter. 
			Right: Rotation number $1/2$ could not correspond to a convex 
		caustic 
			with non empty interior.}\label{fig:deficit}
	\end{figure}
	If a convex caustic $\gamma$ has empty interior then, by Remark 
	\ref{rem:empty}, 
	$\Gamma$ is necessarily an ellipse and the invariant curve 
	corresponding to $\gamma$ has rotation number $\tfrac{1}{2}$ since 
	it contains a diameter.
	On the other hand, if $\intr{\gamma}$  contains some ball of radius 
	$\delta$, then every reflection 
	produces an angle deficit which can be bounded from below by   
	$\delta$ (see Fig.\ref{fig:deficit}).
	Therefore the average number of turns of the billiard trajectory tangent 
	to $\gamma$ is bounded away from 
	$\tfrac 12$. Hence the rotation number is strictly less than 
	$\tfrac 12$. 
\end{proof}	
	 
	Let $\gamma_n$ be a  sequence of convex caustics with the rotation 
	numbers of corresponding invariant curves $\rho_n\in (0;\tfrac{1}{2}]$. 
	By 
	Lemma \ref{lm:ball} we may assume that $\rho_n<\frac{1}{2}$ since 
	otherwise 
	$\gamma_n$ has empty interior and then must be an ellipse by Remark 
	\ref{rem:empty}. Passing 
	to a subsequence we can assume with no loss of generality that 
	$\rho_n$ is strictly increasing,  
	$\rho_n\nearrow \tfrac{1}{2}$.
	
	\begin{lm}
		Let $\gamma_1$ and $\gamma_2$ be two convex caustics for $\Gamma$. If the 
		corresponding invariant curves have rotation numbers 
		$\rho_1<\rho_2$, then
		$\bar\gamma_2\subset\intr{\gamma_1} $. 
	\end{lm}
	
\begin{proof} 
Assume that $\bar\gamma_2$ is not a subset of $\intr{\gamma_1} $. Then 
there are three possibilities: (1) 
$\bar{\gamma_1}\cap\bar{\gamma_2}=\emptyset$;  (2) 
$\gamma_1\cap\gamma_2\ne \emptyset$  or
	(3) $\bar\gamma_1\subset\intr{\gamma_2}$.

In the third case one obviously has $\rho_1\geq \rho_2$ contrary to the 
assumption of the Lemma. In the first and the second cases
there necessarily exists  
 a line supporting to both $\bar\gamma_1$ and 
$\bar\gamma_2$. Therefore, all billiard 
reflections in $\Gamma$ of this line are also supporting 
	lines for both  $\bar\gamma_1$ and 
	$\bar\gamma_2$. 
This means that there exist a whole infinite orbit lying in the intersection of 
the two invariant curves corresponding to $\gamma_1$ and $\gamma_2$. 
But then $\rho_1$ must be equal to $\rho_2$, since 
they are completely 
determined by 
one orbit. 

\end{proof}
\begin{rem}
	The statement of Lemma 1 holds true also in the opposite direction which will not be used below.
	Namely, $\bar\gamma_2\subset\intr{\gamma_1} $ implies 
	$\rho_1<\rho_2$. As we already mentioned in the proof it is obvious 
	that $\rho_1\leq\rho_2$. In addition $\rho_1$ can not be equal to 
	$\rho_2$. Indeed, otherwise there exist two disjoint graphs of $r_1$ and 
	$r_2$ with the same rotation number, invariant 
	under the billiard map of the cylinder, which is impossible since billiard 
	map is a twist map (see for example \cite[p.428]{HK}). 
		
		\end{rem}
Let $\{S_n\}$ be the sequence of string 
parameters corresponding to the caustics $ \gamma_n$. Then by 
Lemma 2,
$S_n$ is decreasing. Denote $S=\lim\limits_{n\to\infty} S_n$. 

	\begin{lm} Boundary of the intersection set 
	$$C=\bigcap\limits_{n=1}^{\infty} 
\bar\gamma_n$$
		is a convex caustic for $\Gamma$ with string parameter $S$. 
	\end{lm}

\begin{proof}
	The intersection set $C$ 
	is compact and convex. Moreover, it is easy to see that $\partial_C$ is 
	also a 
	caustic with string parameter $S$. Indeed, this follows from the 
	following geometric consideration (see Left part of Fig.\ref{fig:deficit}).  
	Fix a 
	point $P$ on 
	$\Gamma$ and consider the cap-bodies 
		$$K_n=\conv(P\cup \bar\gamma_n), \qquad K=\conv(P\cup C).$$
	Then obviously
	$$
	K_n \subseteq K,\qquad  K=\bigcap\limits_{n} K_n,
	$$
	and moreover
	$$
     \leng{\partial _{K_n}}=S_n\rightarrow 
     S=\leng{\partial_K}.
	$$
	In addition, since $\gamma_n$ is a caustic then $S_n$ does not depend on $P\in\Gamma$ (by Theorem 1). Therefore, $S$ also does not depend on $P$, and hence $C$ reconstructs 
	$\Gamma$ via 
string construction. Thus $\partial_C$ is a caustic by Theorem 1.
	\end{proof}
The last step in the proof of the Theorem 2 consists in the following 
Lemma.
\begin{lm}
	The limit caustic $\partial_C$ has empty interior.
\end{lm}
\begin{proof} First notice that it follows from continuity of invariant 
curves and their rotation numbers that the invariant curve corresponding to 
$C$ has rotation number $\tfrac12$. Then from Lemma 1 we conclude that 
$\partial_C$
	has empty interior.
\end{proof}

\newcommand{\f}{\varphi}
\section{Non-smooth caustic. \label{sec:nonsmooth}}
	
	\begin{figure}[!h]
		\centering
		\includegraphics[width=0.5\textwidth]{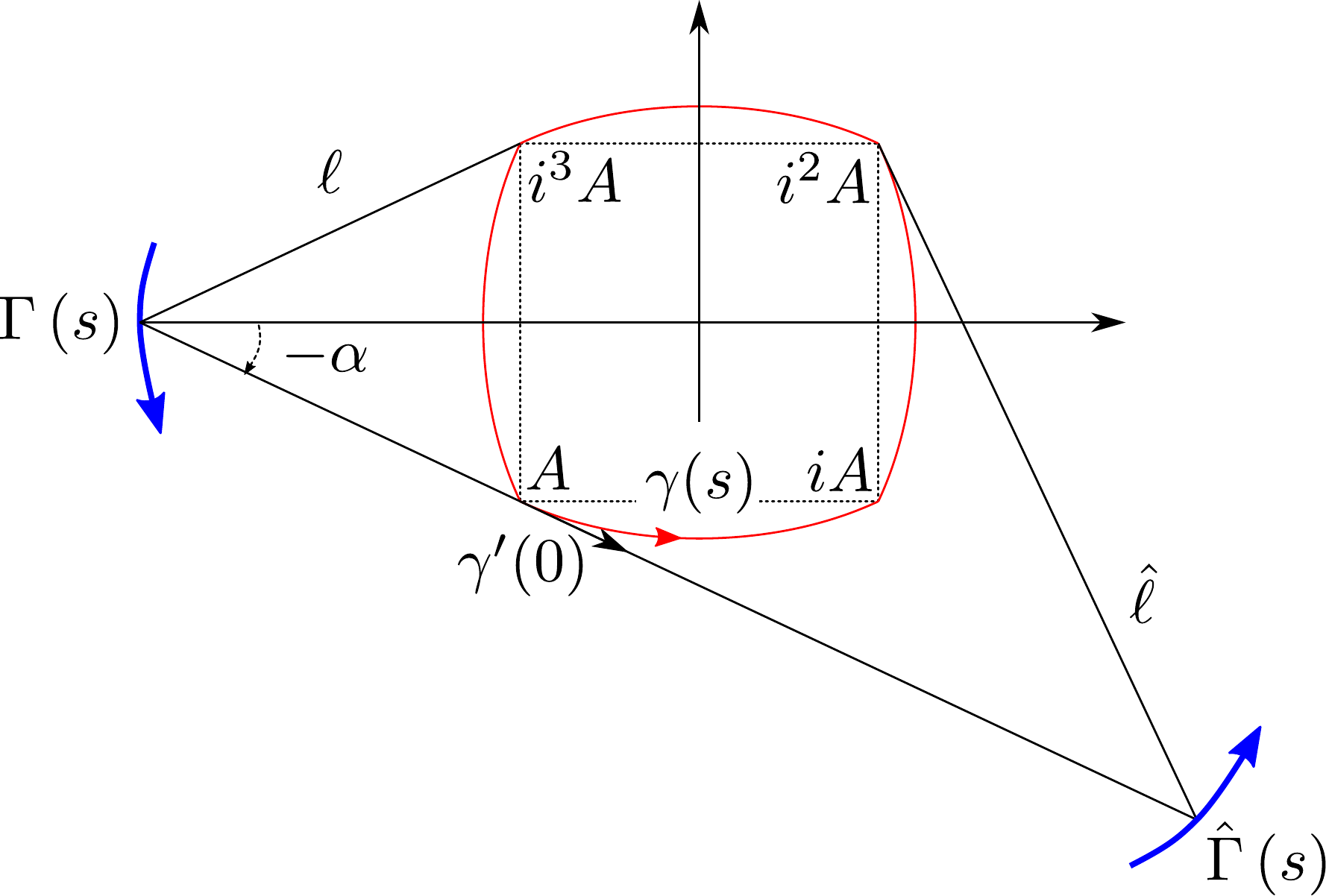}
		\caption{Switched caustic string construction. }\label{fig:table}
	\end{figure}
	Main idea of the proof of our result is to carefully choose the Lazutkin 
	parameter and the germ of function $\gamma$ at the point $A$. While 
	vertex of the string slides in the regime corresponding to the unwinding 
	from $\gamma(s)$, its trajectory corresponds to the smooth curve. Thus 
	we have to take care of the smoothness of $\Gamma$ near only two 
	points 
	corresponding to the  
	switching moments of the first and second kind respectively. We will 
	denote by $\Gamma(s)$ the part of $\Gamma$ corresponding to the 
	switching of the second kind about the point $A$. $\hat{\Gamma}$ will 
	denote the part of $\Gamma$ corresponding to the switching of the first 
	kind about the point $A$. Then the smoothness conditions read as 
	follows: all odd terms in the germs of $\Gamma$ and $\hat{\Gamma}$ 
	have to be orthogonal to the axis of 
	the symmetry while all the even terms has to be collinear with the axis of 
	symmetry.

 \paragraph*{Coordinate formulation.}
Parametrize the curve $\gamma$ by the arc-length parameter 
$s$, so that $|\gamma'(s)|~=~1$. Choose the initial point such 
that $\gamma(0)=A$. Denote by $\alpha$ the angle between $\gamma'(0) 
$ 
and horizontal axis. Then one easily obtains a parametrization for 
$\Gamma$ and $\hat{\Gamma}$ (see Fig. \ref{fig:table}):

\begin{equation}\label{eq:Gamma_parametrization}
\Gamma(s)
=\gamma(s)-t(s)\gamma'(s)
,\qquad\hat{\Gamma}(s)
=\gamma(s)+\hat{t}(s)\gamma'(s)
\end{equation}
where  $t(s)$ ant $\hat{t}(s)$ are some functions of $s$ denoting the length 
of the right part of the string near point $\Gamma(s)$ and left part of the 
string near 
point $\hat{\Gamma}(s)$ correspondingly. Functions $t$ and 
$\hat{t}$ can be found 
from the condition of the string to be unstretchable. We will denote $\i 
A=B$.

\begin{equation}\begin{aligned}
|\Gamma(s)+B|+|t\gamma'(s)|-s=2\L
\\
|\hat{\Gamma}(s)+A|+|\hat{t}\gamma'(s)|+s=2\hat{\L}
\end{aligned}
\end{equation}
where 
$\L=\frac{1}{\sin \alpha}$ and 
$\hat{\L}=\frac{\sqrt{2}}{\sin(\pi/4-\alpha)}$.  Simple computations yield 
for $t(s)$ and 
	$\hat{t}(s)$:
\begin{equation}\label{eq:tfromp}\begin{aligned}
t(s)=
\dfrac{p(s)}{p'(s)},\,\mbox{with }\, p(s)=\dfrac12 
\left((s+2\L)^2-|\gamma(s)+B|^2\right),\\
\hat{t}(s)=-\dfrac{\hat{p}(s)}{\hat{p}'(s)},\,\mbox{with }\,\hat{p}(s)=\dfrac 
12 
\left((s-2\hat{\L})^2-|\gamma(s)+A|^2\right).
\end{aligned}
\end{equation} 

Finally, introducing \eqref{eq:tfromp} into 
\eqref{eq:Gamma_parametrization} we get
\begin{equation}
\label{eq:Gamma_true}
\Gamma(s)=\gamma(s)-\frac{p(s)}{p'(s)}\gamma'(s)
,\qquad \hat{\Gamma}(s)=\gamma(s)-\frac{\hat{p}(s)}{\hat{p}'(s)}
\gamma'(s).
\end{equation}

Orient curve $\gamma$ as it is shown on Fig. \ref{fig:table}. We will use 
complex notation for the coordinates of the points. Then 
smoothness conditions for the $n$-th derivative of $\Gamma$ read 
\begin{equation}\label{eq:smooth}
\Re\left(i^{n-1}\Gamma^{(n)}(0)\right)=0,
\qquad\Re\left(i^{n-1}\hat{\Gamma}^{(n)}(0)\right)=
\Im\left(i^{n-1}\hat{\Gamma}^{(n)}(0)\right).
\end{equation}

For the curve $\gamma(s)$ we get the 
following parametrization: 

\begin{equation}
\label{eq:gamma_param}
\gamma(s)=A+\int\limits_{0}^{s}\exp\left\{i(\varphi(t)-\alpha)\right\}
dt,\qquad
\mbox{ where }\qquad
\varphi(t)=\sum\limits_{n=0}^\infty \f_n t^n\end{equation}

Thus $\f_0=0$,  and $\f_n$ corresponds to the $(n-1)$-st derivative of the 
curvature $\k$.

\begin{lm}
	Smoothness conditions \eqref{eq:smooth} for $n=1$ 
	are always satisfied.
\end{lm}

	The statement of this lemma  follows from the fact that any $C^0$ 
	caustic produces $C^1$ table via string construction. However, we 
	present more analytic proof of this result for a sake of completeness.

\begin{proof}$\phantom{1}$\begin{enumerate}
\item \textbf{Switching of the second kind.}
	From \eqref{eq:Gamma_true} we get
	\begin{equation*}
	\Gamma'=\left(1-\left(\frac{p}{p'}\right)'\right)\gamma'-
	\frac{p}{p'}\gamma''
	\end{equation*}
	therefore conditions \eqref{eq:smooth} read
	\[\Re(p'' \gamma'-p'\gamma'')=0\]
	We will denote $z_1\cdot z_2:=\frac{1}{2}\Re(z_1 \bar{z}_2)$. 
	Using 
	expressions \eqref{eq:tfromp} we get
	\begin{equation*}
	\label{eq:p'}
	p'=-(A+B)\cdot\gamma'+2\L
	,\qquad p''=-(A+B)
	\cdot \gamma''
	\end{equation*}
	 From \eqref{eq:gamma_param} it follows that 
	 $\gamma''=
	i\k \gamma'$ thus $p''\gamma'-p'\gamma''$ can be written as
	\[
	p''\gamma'-p'\gamma''=\frac 12\left(-\Re((A+B)
	\overline{i\k \gamma'})\gamma'+\Re((A+B)\overline{\gamma'})
	(i\k\gamma')-
	4\L i\k\gamma'\right)=i\k (A+B-2\L \gamma').
	\]
	 Thus 
	\[\Re(p''\gamma'-p'\gamma'')=\k \Im(A+B-2\L\gamma')\]
	The latter is identically zero since $\L 
	\gamma'(0)=\Gamma(0)-\gamma(0)$ 
	and so  
	$\Im(\L \gamma')=\Im (A)$ (see Fig. \ref{fig:table}).
	
	\item \textbf{Switching of the first kind.}
	 Similarly, smoothness conditions \eqref{eq:smooth}
reads	
	\[\Re (\hat{p}''\gamma'-\hat{p}'\gamma'')=
	\Im(\hat{p}''\gamma'-\hat{p}'\gamma'')\]
	where 
	\begin{equation*}
	\hat{p}'=-(2A)\cdot \gamma'-2\hat{\L}
	,\qquad\hat{p}''=-(2A)\cdot\gamma''
	\end{equation*}
	and so 
	\[\hat{p}''\gamma'-\hat{p}'\gamma''=\left( 
	\Re(Ai\k\overline{\gamma'})\gamma'+
	\Re(A\overline{\gamma'})(i\k\gamma')+
	2\hat{\L}i\k\gamma'\right)=2i\k (A +\hat{\L}\gamma')
	\]
	Real part of the right-hand side of the latter is always equal to the 
	imaginary part by the definition of $\hat{\L}$.
\end{enumerate}
\end{proof}

Two conditions \eqref{eq:smooth} for 
$n=2$ provide, via the computations similar to the above, two equations 
for parameters $\f_1$ and $\f_2$ with 
coefficients depending on $\alpha$. 

\begin{equation*}
\label{eq:n=2}
\begin{aligned}
&\dfrac{\f_1^2\sin\alpha 
-\f_1\sin\alpha\cos\alpha-\f_2\cos\alpha}{\sin\alpha \cos^2 \alpha} = 
0\\
&\dfrac{ \f_1 (\cos 2 \alpha + 2 ( \sin\alpha-\cos\alpha ) \f_1)  
-2 (\cos \alpha + \sin \alpha) \f_2}{(\cos \alpha- \sin \alpha) (1 + \sin 2 
\alpha)} = 0.
\end{aligned}
\end{equation*}
The latter system 
has 
a solution 
\begin{equation}
\label{eq:C2germ}\f_1=\frac{1}{2} \cos \alpha (1+\sin 2\alpha),\qquad 
\f_2= -\frac 18 
\cos^2 2 
\alpha \sin 2 \alpha,\end{equation}
which provides  a family of germs for $\gamma$, depending on 
parameter 
$\alpha$, guaranteeing the $C^2$-smoothness for the table  $\Gamma$.

Next we will need to construct the whole curve $\gamma$ providing the 
needed phenomenon in the string construction. Recall that our 
geometric idea was based on the construction of the curve 
$\gamma_0$ (see Fig.\ref{fig:table_intro}). Thus we need to present a 
convex curve of length $\ss$, starting at $A$ and ending at $iA$, having 
tangent slope $-\alpha$ at the left end and being symmetric with respect 
to the vertical axis.  We define $\gamma$ from $\f$ through 
\eqref{eq:gamma_param}. In order to finish the construction we have to 
prove the following theorem.
\begin{thm}
	There exists a strictly monotonically increasing function $\varphi(s)$ 
	satisfying 
	the following conditions
	\begin{enumerate}
		\item $\varphi(s)$ has the given germ \eqref{eq:C2germ} at $s=0$.
		\item $\varphi_0(\ss/2)=\alpha$ and $\f_{2n}(\ss/2)=0$ 
		for $n\geqslant 1$.
		\item \label{cond:integral} $\int\limits_0^{\ss/2} 
		\cos\varphi(s)ds=1$.
	\end{enumerate}
\end{thm}

\begin{proof}
	Thanks to  Borel theorem there exist a set $\Psi$ of $C^\infty$ 
	functions having given germs at $s=0$ and $s=\ss/2$.  Since for 
	$\alpha<\frac{\pi}{2}$ term $\f_1$ in \eqref{eq:C2germ} is positive, one 
	may assume without loss of generality that $\Psi$ consists of 
	strictly monotonically increasing functions. Therefore 
	the only condition which has to be satisfied is the
	condition \ref{cond:integral}.
	Taking small enough $\varepsilon$-step in $s$ we can assure
	$\psi(\varepsilon)<\tfrac{\alpha}{100}$ for 
	all $\psi\in \Psi$.  

	\begin{figure}[hbt]
	\centering
	\includegraphics[width=0.5\textwidth]{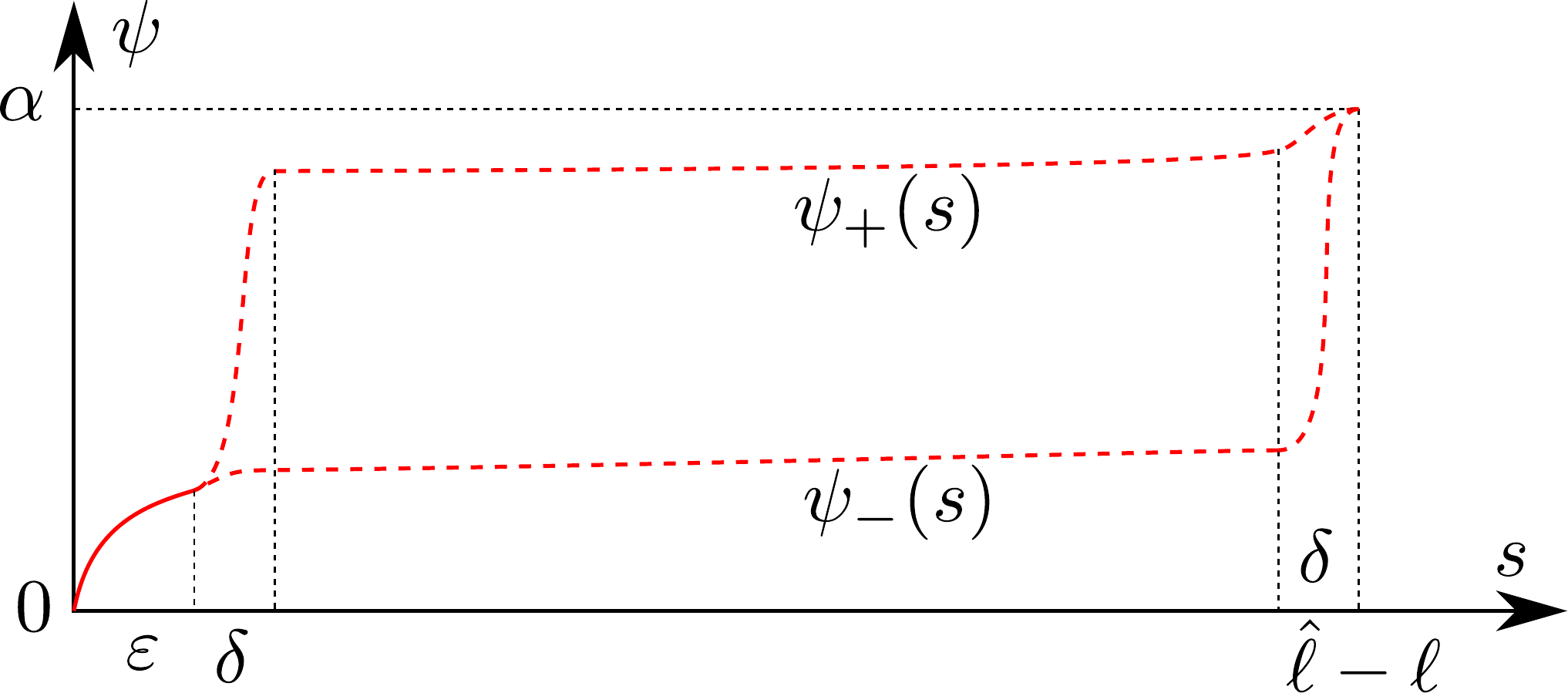}
	\caption{Construction of the solution.}\label{fig:angle}
\end{figure}
	Next we choose 
	two functions $\psi_-$ and $\psi_+$ from the set $\Psi$ as on Fig. 
	\ref{fig:angle}. That is $\psi_+(s)$ almost 
	equals to $\alpha$ for $s\in (\varepsilon+\delta,\ss/2-\delta)$ and 
	$\psi_-(s)$ is almost equal to $\psi(\varepsilon)$ for 
	$s\in(\varepsilon,\ss/2-\delta)$ for small enough $\delta$.  
	We will look for $\f$ as a convex 
	combination 
	$\f(s)=l\psi_-(s)+(1-l)\psi_+(s)$.
	Therefore $\f(s)$ obviously satisfies conditions 1 and 2. 
	 
	If we may choose $\psi_\pm$ in such a way that
	\begin{equation}
	\label{eq:phi1}
	(\ss/2)\cos \alpha<
	\int\limits_0^{\ss/2} \cos(\psi_-(s)-\alpha) ds<1
	\qquad \mbox{and}\qquad \ss/2>
	\int\limits_0^{\ss/2} 
	\cos(\psi_+(s)-\alpha) 
	ds>1
	\end{equation}
than there exists such 
	 $l$ that 
	$\int\limits_0^{\ss/2}\cos (\f(s))ds=1$, thus satisfying condition 3. 	
Hence it is sufficient to check that conditions \eqref{eq:phi1} can be 
satisfied 
	for an open set of parameters $\alpha$. Recall, 
	that by the construction $\ss=2\hat{\ell}-2\ell$.  From the first 
	inequality \eqref{eq:phi1}
	 we obtain, since 
	 $\alpha<\frac{\pi}{4}$,
	\[\hat{\L}-\L =\frac{2}{\cos\alpha-\sin\alpha}-\frac{1}{\sin\alpha}
	<\frac{1}{\cos\alpha}.\]
	
	This condition can be interpreted as follows: \emph{length of the curve 
		$\gamma$ could not exceed the sum of lengths of the segments of 
		two tangent lines from point $P$ to $\gamma$ (see Fig. 
		\ref{fig:table_intro})}. The latter 
		inequality is satisfied whenever
	$\tan 2\alpha<1$ or \begin{equation}
	\label{eq:alpha}
	\alpha<\frac{\pi}{8}
	\end{equation}
	
	Second condition in \eqref{eq:phi1} has the following geometric 
	interpretation: 
	\emph{length of 
	$\gamma$ could not be less than the distance between points $A$ and 
	$B$.} This yields:
	\[3\sin \alpha-\cos\alpha
	> \cos\alpha\sin\alpha-\sin^2\alpha\]
	Since the latter is satisfied for $\alpha=\frac{\pi}{8}$ we have found an 
	open set of $\alpha$ for which one could find appropriate functions 
	$\psi_-$ and $\psi_+$ shown on 
	Fig.\ref{fig:angle}. 
	
\end{proof}

\begin{rem}
	Since conditions \eqref{eq:smooth} provides two conditions on $\f_n$  
	to obtain $C^3$ of $\Gamma$ one gets four equations for $\f_1$, 
	$\f_2$, $\f_3$ and $\alpha$. Yet the number of parameters match the 
	number of equations, the corresponding value of $\alpha$ violates 
	inequality (9). Since inequality (9) arise from the construction based on 
	square symmetry, there is a hope that starting from other regular 
	polygons one can obtain inequality which can be satisfied. However, we 
	haven't found such examples. 
\end{rem}

\section{Open problems.}
Here we want to stress the general questions which are ultimately
related to the string construction.  Since the string construction is very 
implicit these questions turn out to be non-trivial.
\begin{qst}
	Is it possible to have two convex caustics $\gamma_1$ and $\gamma_2$
	of $\Gamma$ such that none of them is a subset of the interior of the 
	other?
	
\end{qst}

\begin{figure}[!h]
	\centering
	\includegraphics[width=0.45\textwidth]{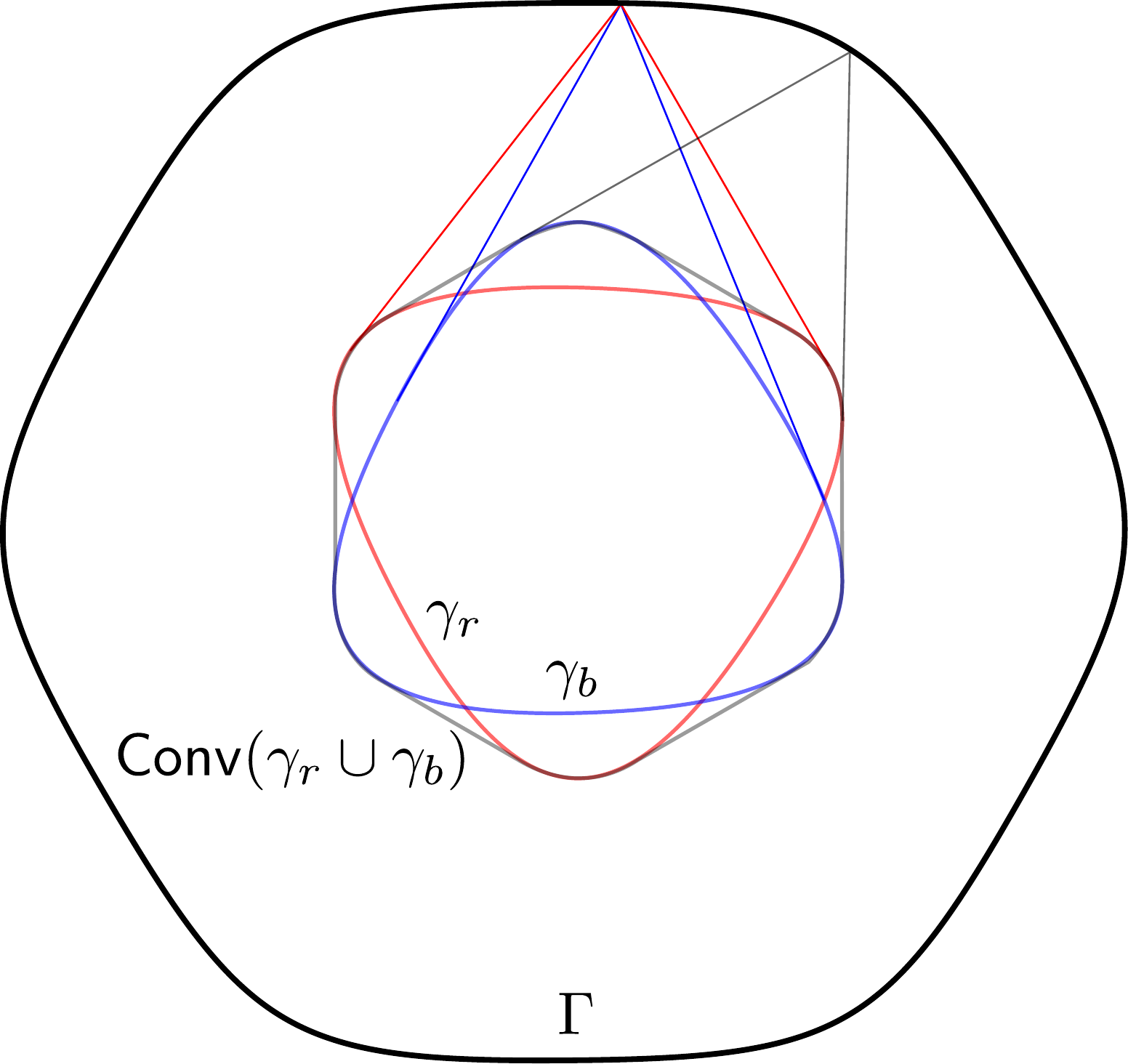}
	\caption{Convex hull of two intersecting caustics is also a caustic. 
	}\label{fig:magendavid}\end{figure} 

In such a case $\gamma_1$ and $\gamma_2$ must have the same rotation 
number since there is a line tangent to both of the caustics.
Moreover it is obvious that $\bar\gamma_1$ and $\bar\gamma_2$ 
cannot be disjoint. So the question is if it is possible that two convex 
caustics
have non-trivial intersection. In such a case also their convex hull is also a 
caustic. One can strengthen the question:

\begin{qst}
	 Is it possible for $\Gamma$ which is symmetric with 
respect to certain axis to have a convex caustic $C$ which is not symmetric 
with respect to this axis. 
\end{qst}

For example one could imagine two caustics forming 
rounded  David Star (Fig \ref{fig:magendavid}). Let us remark that the 
answer to the quantum analog of this question is 
positive: for symmetric domain Dirichlet eigenfunction can be 
non-symmetric. We could not however decide 
if such a counterexample would be possible in the original setting.

\begin{qst}
	How irregular a convex caustic can be versus regular boundary curve 
$\Gamma$?
\end{qst}

\begin{qst}
	Let  $\Gamma$ be a billiard table different from circle having a convex caustic 
$\gamma$. For every point $P\in\Gamma$ denote $P_-, P_+$ the points on the caustic $\gamma$
which are tangency points on the tangent lines to $\gamma$ passing through $P$.
Is it possible that the length of the arc of $\gamma$ between $P_-, P_+$ 
is
constant not depending on $P$?
\end{qst}

	\bibliographystyle{plain}
		\bibliography{switched}
		
		

\end{document}